\documentclass{article}

\usepackage{amsmath,amssymb}
\usepackage{amsthm}
\usepackage{url}

\usepackage{algorithm}
\usepackage{algpseudocode}

\newtheorem{theorem}{Theorem}
\newtheorem{corollary}{Corollary}

\theoremstyle{remark}
\newtheorem{example}{Example}

\begin{document}

\title{A convolution for the complete and\\ elementary symmetric functions}
\author{Mircea Merca\footnote{mircea.merca@profinfo.edu.ro}
	\\ 
	\small Department of Mathematics,
	\small University of Craiova, 
	\small Craiova, 200585 Romania}
\date{}
\maketitle

\begin{abstract}
In this paper we give a convolution identity for the complete and elementary symmetric functions. 
This result can be used to proving and discovering some combinatorial identities involving $r$-Stirling numbers, $r$-Whitney numbers and $q$-binomial coefficients. As a corollary we derive a generalization of the quantum Vandermonde's convolution identity.
\end{abstract}

\section{Introduction.}

Let $n$ be a positive integer. Being given a set of variables $\{x_1,x_2,\ldots,x_n\}$, recall \cite{Mac95} that the $k$th elementary symmetric function $e_k (x_1,x_2,\dots,x_n)$ and the $k$th complete homogeneous symmetric function $h_k (x_1,x_2,\dots,x_n)$ on these variables are given, respectively, by
\begin{eqnarray*}
	e_k (x_1,x_2,\ldots,x_n)=\sum _{1\le i_1 <i_2<\ldots <i_k\le n }x_{i_1} x_{i_2} \ldots x_{i_k}\ ,\\
	h_k (x_1,x_2,\ldots,x_n)=\sum _{1\le i_1\le i_2\le\ldots\le i_k\le n }x_{i_1} x_{i_2} \ldots x_{i_k}\ ,
\end{eqnarray*}
for $k=1,2,\ldots,n$. We set 
$e_0(x_1,\ldots,x_n)=1$ and $h_0(x_1,\ldots,x_n)=1$ 
by convention.
For $k>n$ or $k<0$, we set $e_k(x_1,\ldots,x_n)=0$ and $h_k(x_1,\ldots,x_n)=0$.

A composition of $n$ is a way of writing $n$ as the sum of positive integers, i.e.,
$$n = \lambda_1 + \lambda_2 + \cdots + \lambda_k\ .$$
If $\lambda_1\ge \lambda_2 \ge \cdots \ge \lambda_k$ then, this representation is known as an integer partition. In order to indicate that $\lambda=[\lambda_1,\lambda_2,\ldots,\lambda_k]$ is a partition of $n$, we use the notation $\lambda \vdash n$, introduced by Andrews in \cite{Andrews76}. The number of parts of $\lambda$ will be called the length of $\lambda$ and denoted by $l(\lambda)=k$. For each $i$ ($1\le i \le n$), the number of times that $i$ appears as a part of $\lambda$ is called the multiplicity of $i$ in $\lambda$, denoted $t_i(\lambda)$. An alternate way to write partitions down then is by the standard notation $$\lambda=[1^{t_1(\lambda)},2^{t_2(\lambda)},\ldots,n^{t_n(\lambda)}]\ .$$ 
Clearly
$$l(\lambda)=t_1(\lambda)+t_2(\lambda)+\cdots+t_n(\lambda)\ .$$
For each partition $\lambda$, we note
$$f_{\lambda}(x_1,\ldots,x_n)=\prod_{i\ge 1}f_i^{t_i(\lambda)}(x_1,\ldots,x_n)\quad\ ,$$
where $f$ is any of these complete or elementary symmetric functions.

In this paper, we shall prove:

\begin{theorem}\label{T:1}
	Let $k$, $m$ and $n$ be three positive integers and let $[\lambda_1,\lambda_2,\ldots,\lambda_m]$ be a composition of $n$. Then
	\begin{equation}\label{e:1.1}
	f_k(x_1,x_2,\ldots,x_n)=\sum_{k_1+k_2+\cdots+k_m=k}\prod_{i=1}^{m}f_{k_i}(x_{a_{i-1}+1},\ldots,x_{a_i})\ ,
	\end{equation}
	where 
	$$a_0=0\ ,\qquad a_i=\lambda_1+\lambda_2+\cdots+\lambda_i\ ,$$
	$x_1,x_2,\ldots,x_n$ are independent variables and $f$ is any of these complete or elementary symmetric functions.
\end{theorem}

It is clear that $f_{k_i}(x_{a_{i-1}+1},\ldots,x_{a_i})=0$ when $k_i>\lambda_i$ or $k_i<0$. For $k_i=0$, we have $f_{k_i}(x_{a_{i-1}+1},\ldots,x_{a_i})=1$. If $k_1,k_2,\ldots,k_m$ are nonnegative integers such that $k_1+k_2+\cdots+k_m=k$, then $[k_1+1,k_2+1,\ldots,k_m+1]$ is a composition of $k+m$ into exactly $m$ parts. Therefore, we can say that the right side of \eqref{e:1.1} is a summation over all compositions
$$[k_1+1,k_2+1,\ldots,k_m+1]$$
of $k+m$ into exactly $m$ parts with the property 
$$0\le k_i\le\lambda_i\ ,\quad \mbox{for}\quad i=1,2,\ldots,m\ .$$

\begin{example}  Taking into account that $[2,1,2]$ is a composition of $5$, for $n=5$ and  $m=3$, we have
	\begin{eqnarray*}
		&&f_3(x_1,x_2,x_3,x_4,x_5) = f_1(x_3)f_2(x_4,x_5)+f_1(x_1,x_2)f_2(x_4,x_5)\\
		&&\qquad+f_1(x_1,x_2)f_1(x_3)f_1(x_4,x_5)+f_2(x_1,x_2)f_1(x_4,x_5)+f_2(x_1,x_2)f_1(x_3)\ ,
	\end{eqnarray*}
	where $f$ is any of these complete or elementary symmetric functions.
\end{example}

\begin{corollary}\label{C:1.1}
	Let $k$, $m$ and $n$ be three positive integers and let $x_1,x_2,\ldots,x_n$ be $n$ independent variables. Then
	\begin{equation*}
	f_k(\underbrace{x_1,\ldots,x_n,\ldots,x_1,\ldots,x_n}_{m}) = \sum_{\substack{\lambda\vdash k\\l(\lambda)\le m}}\binom{m}{m-l(\lambda),\lambda}f_{\lambda}(x_1,\ldots,x_n)\ ,
	\end{equation*}
	where 
	$$\binom{m}{m-l(\lambda),\lambda}=\frac{m!}{(m-l(\lambda))!t_1(\lambda)!\cdots t_k(\lambda)!}\ $$
	is the multinomial coefficient and $f$ is any of these complete or elementary symmetric functions.
\end{corollary}

This result is immediate from Theorem \ref{T:1}. We note that Corollary \ref{C:1.1} is the case 
$$\left\lbrace x_{a_0+1},\ldots,x_{a_1} \right\rbrace = \left\lbrace x_{a_1+1},\ldots,x_{a_2}\right\rbrace = \cdots = \left\lbrace x_{a_{m-1}+1},\ldots,x_{a_m} \right\rbrace\ $$
in Theorem \ref{T:1}.

\begin{example} For $m=3$ and $k=4$, we have
	\begin{eqnarray*}
		&&f_4(x_1,\ldots,x_n,x_1,\ldots,x_n,x_1,\ldots,x_n) =3f_1^2(x_1,\ldots,x_n)f_2(x_1,\ldots,x_n)\\
		&&\qquad+6f_1(x_1,\ldots,x_n)f_3(x_1,\ldots,x_n)+3f_2^2(x_1,\ldots,x_n)+3f_4(x_1,\ldots,x_n)\ ,
	\end{eqnarray*}
	where $f$ is any of these complete or elementary symmetric functions.
\end{example}

By the relation
$$
\prod_{i=1}^{n}(x+x_i)(x-x_i)=\prod_{i=1}^{n}(x^2-x_i^2)\ ,
$$
we deduce that
$$
e_{2k}(x_1,\ldots,x_n,-x_1,\ldots,-x_n)=(-1)^ke_k(x_1^2,\ldots,x_n^2)\ 
$$
and
$$
e_{2k+1}(x_1,\ldots,x_n,-x_1,\ldots,-x_n)=0\ .
$$
Taking into account that
$$e_k(-x_1,\ldots,-x_n)=(-1)^ke_k(x_1,\ldots,x_n)\ ,$$
the following result is a consequence of Theorem \ref{T:1}.

\begin{corollary}\label{C:1.2}
	Let $k$ and $n$ be two positive integers. Then
	$$e_k(x_1^2,\ldots,x_n^2)=\sum_{i=-k}^{k}(-1)^ie_{k+i}(x_1,\ldots,x_n)e_{k-i}(x_1,\ldots,x_n)\ .$$
\end{corollary}

We note that Corollary \ref{C:1.2} is the case $n=2$ of the generalized Girard-Waring formula \cite{Kon96,Zen97}.
Some applications of this corollary was recently published by Merca \cite{Merca12}.
The generalized Girard-Waring formula can be used to express the monomial
symmetric functions with equal exponents in terms of elementary symmetric functions. 

In this paper, we use Theorem \ref{T:1} to proving and discovering some combinatorial identities involving $r$-Stirling numbers, $r$-Whitney numbers and $q$-binomial coefficients. This is possible because these special numbers are specializations of complete and elementary symmetric functions.
For instance, the well-known Vandermonde's convolution 
$$\sum_{i=0}^{k}\binom{t}{i}\binom{n-t}{k-i}=\binom{n}{k}$$
and its generalization
$$\sum_{k_1+k_2+\cdots+k_m=k}\binom{\lambda_1}{k_1}\binom{\lambda_2}{k_2}\cdots \binom{\lambda_m}{k_m}=\binom{\lambda_1+\lambda_2+\cdots+\lambda_m}{k}\ $$
or the similar convolution
$$\sum_{k_1+k_2+\cdots+k_m=k}\prod_{i=1}^{m}\binom{\lambda_i+k_i-1}{k_i}=\binom{\lambda_1+\lambda_2+\cdots+\lambda_m+k-1}{k}\ $$
are very special cases of this theorem. The $q$-analogues of these generalizations are also obtained as specializations of our theorem. For $r<t\le k<n$, Broder \cite{Broder84} proved two identities involving $r$-Stirling numbers of both kind:
\begin{equation}\label{e:1.2}
\sum_{i=0}^{n-k}\begin{bmatrix}t\\t-i\end{bmatrix}_{r}\begin{bmatrix}n\\k+i\end{bmatrix}_{t}=\begin{bmatrix}n\\k\end{bmatrix}_r
\end{equation}
and
\begin{equation}\label{e:1.3}
\sum_{i=0}^{n-k}\begin{Bmatrix}t+i\\t\end{Bmatrix}_{r}\begin{Bmatrix}n-i\\k\end{Bmatrix}_{t+1}=\begin{Bmatrix}n\\k\end{Bmatrix}_r .
\end{equation}
In this paper, these identities are very special cases of more general identities 
involving $r$-Whitney numbers of both kind.
As far as we know, these general identities are new.

\section{Proof of Theorem \ref{T:1}}

According to \cite{Mac95}, we have
$$\sum_{k\geq 0} e_k(x_1,\ldots,x_n) z^k=\prod_{i=1}^{n}(1+x_iz)$$ 
and
$$\sum_{k\geq 0} h_k(x_1,\ldots,x_n) z^k=\prod_{i=1}^{n}(1-x_iz)^{-1}\ .$$ 
Then, using the well-known Cauchy products of two power series, we can write
\begin{align*}
&\sum_{k\geq 0} e_k(x_1,\ldots,x_n,y_1,\ldots,y_t) z^k\\
&\qquad= \prod_{i=1}^{n}(1+x_iz)\prod_{i=1}^{t}(1+ y_iz)\\
&\qquad=\left(\sum_{k\geq 0} e_k(x_1,\ldots,x_n) z^k \right) \left(\sum_{k\geq 0} e_k(y_1,\ldots,y_t) z^k \right) \\
&\qquad= \sum_{i,j\geq 0} e_i(x_1,\ldots,x_n)e_j(y_1,\ldots,y_t) z^{i+j}\ 
\end{align*} 
and
\begin{align*}
&\sum_{k\geq 0} h_k(x_1,\ldots,x_n,y_1,\ldots,y_t) z^k\\
&\qquad= \prod_{i=1}^{n}(1-x_iz)^{-1}\prod_{i=1}^{t}(1-y_iz)^{-1}\\
&\qquad=\left(\sum_{k\geq 0} h_k(x_1,\ldots,x_n) z^k \right) \left(\sum_{k\geq 0} h_k(y_1,\ldots,y_t) z^k \right) \\
&\qquad= \sum_{i,j\geq 0} h_i(x_1,\ldots,x_n)h_j(y_1,\ldots,y_t) z^{i+j}\ .
\end{align*} 
Extracting coefficients of $z^k$ we get
\begin{equation}\label{eq:2.1}
f_k(x_1,\ldots,x_n,y_1,\ldots,y_t ) = \sum_{i=0}^{k}f_{k-i}(x_1,\ldots,x_n)f_{i}\left(y_1,\ldots,y_t \right)\ ,
\end{equation}
where $f$ is any of these complete or elementary symmetric functions.

We proceed to prove Theorem \ref{T:1} by successive application of the relation \eqref{eq:2.1}. 
For $1\le i \le m$, we denote by $B_i$ the set $\left\lbrace x_{a_{i-1}+1},\ldots,x_{a_i}\right\rbrace $.
It is clear that $\{B_i\}_{1\le i \le m}$ is a set partition of $\{x_1,x_2,\ldots,x_n\}$.
Thus, using the notation
$$f(B_i)=f(x_{a_{i-1}+1},\ldots,x_{a_i})\ ,$$
we can write
\begin{align*}
&f(x_1,\ldots,x_n)=f_k(B_1\cup\cdots\cup B_{m}) \\
&= \sum_{k_m=0}^{k}f_{k-k_m}(B_1\cup\cdots\cup B_{m-1})f_{k_m}(B_m)\\
&= \sum_{k_m=0}^{k}\sum_{k_{m-1}=0}^{k-k_m}f_{k-k_m-k_{m-1}}(B_1\cup\cdots\cup B_{m-2})f_{k_{m-1}}(B_{m-1})f_{k_m}(B_m)\ .
\end{align*}
Therefore, we conclude that
\begin{align*}
&f_k(B_1\cup\cdots\cup B_{m}) \\
&= \sum_{k_m=0}^{k}\sum_{k_{m-1}=0}^{k-k_m}\cdots \sum_{k_2=0}^{k-k_m-\cdots- k_3}f_{k-k_m-\cdots-k_2}(B_1)f_{k_2}(B_2)\cdots f_{k_m}(B_m)\\
&=\sum_{k_1+k_2+\cdots+k_m=k}f_{k_1}(B_1)f_{k_2}(B_2)\cdots f_{k_m}(B_m)
\end{align*}
and Theorem \ref{T:1} is proved.

\section{r-Stirling and r-Whitney numbers}

The $r$-Stirling numbers were introduced into the literature by Broder \cite{Broder84} in 1984 and represent a certain generalization of the classical Stirling numbers. 
The $r$-Stirling numbers of the first kind
$$\begin{bmatrix}n\\k\end{bmatrix}_r$$
count restricted permutations and are defined, for any positive integer $r$, as the number of permutations of the set $\{1,2,\ldots,n\}$ that have $k$ cycles such that the numbers $1,2,\ldots,r$ are in distinct cycles. 
The $r$-Stirling numbers of the second kind
$$\begin{Bmatrix}n\\k\end{Bmatrix}_r$$
are defined as the number of partitions of the set $\{1,2,\ldots,n\}$ into $k$ non-empty disjoint subsets, such that the numbers $1,2,\ldots,r$ are in distinct subsets.
It is clear that the case $r=1$ gives the classical unsigned Stirling numbers.
Many properties of $r$-Stirling numbers are presented in \cite{Kuba10,Mezo08,Mezo08a}.

According to \cite{Broder84},  the $r$-Stirling numbers of the first kind are the elementary symmetric functions of the numbers $r,\ldots,n$, i.e.,
$$
\begin{bmatrix}n+1\\n+1-k\end{bmatrix}_r=e_k(r,\ldots,n)\ 
$$
and the $r$-Stirling numbers of the second kind are the complete homogeneous symmetric functions of the numbers $r,\ldots,n$, i.e.,
$$
\begin{Bmatrix}n+k\\n\end{Bmatrix}_r=h_k(r,\ldots,n)\ .
$$

In 1973, Dowling \cite{Dowling73} constructed and studied a class of geometric lattices over a finite group $G$ of order $p>0$. For further information on lattices, see \cite{Beno96,Dowling73,Stanley97}.
The Whitney numbers of the first kind of Dowling lattices, $w_p(n,k)$, are given by
\begin{equation}\label{eq:3.1}
p^n\left(x\right)_n= \sum_{k=0}^{n}w_p(n,k)(px+1)^k
\end{equation}
and the Whitney numbers of the second kind, $W_p(n,k)$, are given by
\begin{equation*}\label{eq:4.2}
(px+1)^n=\sum_{k=0}^{n}p^k W_p(n,k)\left(x\right)_k\ ,
\end{equation*}
where $(x)_n$ is the falling factorial, i.e.,
$$(x)_n=x(x-1)\cdots(x-n+1)\ ,$$
with $(x)_0=1$. Many properties of Whitney numbers and their combinatorial interpretations can be seen in \cite{Beno96,Beno97,Beno99,Dowling73}.
By \eqref{eq:3.1}, we deduce that
$$w_{p}(n+1,n+1-k)=(-1)^ke_k(1,p+1,2p+1,\ldots,np+1)$$
and according to Benoumhani \cite[Corollary 4]{Beno96}, we have
$$W_{p}(n+k,n)=h_k(1,p+1,2p+1,\ldots,np+1)\ .$$

The $r$-Whitney numbers were introduced in 2010 by Mez\H{o} \cite{Mezo10} as a new class of numbers generalizing the $r$-Stirling and Whitney numbers. According to \cite{Cheon12,Mezo10}, the $n$th power of $px+r$ can be expressed in terms of the falling factorial as follows
\begin{equation*}
(px+r)^n=\sum_{k=0}^{n}p^{k}W_{p,r}(n,k)(x)_k\ ,
\end{equation*}
where the coefficients $W_{p,r}(n,k)$ are called $r$-Whitney numbers of the second kind. The $r$-Whitney numbers of the first kind are the coefficients of $(px+r)^k$ in the reverse relation
\begin{equation*}
p^n(x)_n=\sum_{k=0}^{n}w_{p,r}(n,k)(px+r)^k\ .
\end{equation*}
It is clear that the case $r=1$ gives the Whitney numbers of Dowling lattices. 
On the other hand, it is an easy exercise to show that 
$$w_{p,r}(n+1,n+1-k)=(-1)^ke_k(r,p+r,2p+r,\ldots,np+r)$$
and
$$W_{p,r}(n+k,n)=h_k(r,p+r,2p+r,\ldots,np+r)\ .$$
Therefore, the following two results are immediate from Theorem \ref{T:1}.

\begin{corollary}\label{C3.1}
	Let $k$, $m$, $n$, $p$ and $r$ be five positive integers and let $[\lambda_1,\lambda_2,\ldots,\lambda_m]$ be a composition of $n$. Then
	$$w_{p,r}(n,n-k)=\sum_{k_1+\cdots+k_m=k}\prod_{i=1}^{m}w_{p,a_{i-1}p+r}(\lambda_i,\lambda_i-k_i)\ ,$$
	where 
	$$a_0=0 \qquad\mbox{and}\qquad a_i=\lambda_1+\lambda_2+\cdots+\lambda_i\ .$$
\end{corollary}

\begin{proof}
	We have
	\begin{align*}
	& e_{k_i}(a_{i-1}p+r,(a_{i-1}+1)p+r,\ldots,(a_i-1)p+r) \\
	&\qquad = e_{k_i}(a_{i-1}p+r,a_{i-1}p+r+p,\ldots,a_{i-1}p+r+(a_i-a_{i-1}-1)p) \\
	&\qquad = (-1)^{k_i}w_{p,a_{i-1}p+r}(a_{i}-a_{i-1},a_i-a_{i-1}-k_i)\\
	&\qquad = (-1)^{k_i}w_{p,a_{i-1}p+r}(\lambda_{i},\lambda_i-k_i)\ .
	\end{align*}
	According to Theorem \ref{T:1}, the corollary is proved.
\end{proof}

\begin{corollary}\label{C3.2}
	Let $k$, $m$, $n$, $p$ and $r$ be five positive integers and let $[\lambda_1,\lambda_2,\ldots,\lambda_m]$ be a composition of $n+1$. Then
	$$W_{p,r}(n+k,n)=\sum_{k_1+\cdots+k_m=k}\prod_{i=1}^{m}W_{p,a_{i-1}p+r}(\lambda_i+k_i-1,\lambda_i-1)\ ,$$
	where 
	$$a_0=0 \qquad\mbox{and}\qquad a_i=\lambda_1+\lambda_2+\cdots+\lambda_i\ .$$
\end{corollary}

\begin{proof}
	We have
	\begin{align*}
	& h_{k_i}(a_{i-1}p+r,(a_{i-1}+1)p+r,\ldots,(a_i-1)p+r) \\
	&\qquad = h_{k_i}(a_{i-1}p+r,a_{i-1}p+r+p,\ldots,a_{i-1}p+r+(a_i-a_{i-1}-1)p) \\
	&\qquad = W_{p,a_{i-1}p+r}(a_{i}-a_{i-1}-1+k_i,a_i-a_{i-1}-1)\\
	&\qquad = W_{p,a_{i-1}p+r}(\lambda_{i}-1+k_i,\lambda_i-1)\ .
	\end{align*}
	According to Theorem \ref{T:1}, the corollary is proved.
\end{proof}

We remark that, the case $m=2$ in Corollaries \ref{C3.1} and \ref{C3.2} can be written as

\begin{corollary}\label{C:3.3}
	Let $k$, $n$, $p$, $r$  and $t$ be five positive integers. Then
	$$w_{p,r}(n,k)=\sum_{i=0}^{n-k}w_{p,r}(t,t-i)w_{p,tp+r}(n-t,k-t+i).$$
\end{corollary}

\begin{corollary}\label{C:3.4}
	Let $k$, $n$, $p$, $r$  and $t$ be five positive integers. Then
	$$W_{p,r}(n,k)=\sum_{i=0}^{n-k}W_{p,r}(t-1+i,t-1)W_{p,tp+r}(n-t-i,k-t).$$
\end{corollary}

We note that the $r$-Whitney numbers of both kinds may be
reduced to the $r$-Stirling numbers of both kinds by setting $p = 1$, i.e.,
$$w_{1,r}(n,n-k)=(-1)^k\begin{bmatrix}n+r\\n+r-k\end{bmatrix}_r$$
and
$$W_{1,r}(n+k,n)=\begin{Bmatrix}n+r+k\\n+r\end{Bmatrix}_r\ .$$
Thus, the case $p=1$ in Corollaries \ref{C3.1} and \ref{C3.2} can be written as

\begin{corollary}\label{C:3.5}
	Let $k$, $m$, $n$ and $r$ be four positive integers such that $r<n$ and let $[\lambda_1,\lambda_2,\ldots,\lambda_m]$ be a composition of $n-r$. Then
	$$\begin{bmatrix}n\\n-k\end{bmatrix}_r=\sum_{k_1+\cdots+k_m=k}\prod_{i=1}^{m}\begin{bmatrix}r+a_i\\r+a_i-k_i\end{bmatrix}_{r+a_{i-1}}\ ,$$
	where 
	$$a_0=0 \qquad\mbox{and}\qquad a_i=\lambda_1+\lambda_2+\cdots+\lambda_i\ .$$
\end{corollary}

\begin{corollary}\label{C:3.6}
	Let $k$, $m$, $n$ and $r$ be four positive integers such that $r\le n$ and let $[\lambda_1,\lambda_2,\ldots,\lambda_m]$ be a composition of $n+1-r$. Then
	$$\begin{Bmatrix}n+k\\n\end{Bmatrix}_r=\sum_{k_1+\cdots+k_m=k}\prod_{i=1}^{m}\begin{Bmatrix}r+a_i-1+k_i\\r+a_i-1\end{Bmatrix}_{r+a_{i-1}}\ ,$$
	where 
	$$a_0=0 \qquad\mbox{and}\qquad a_i=\lambda_1+\lambda_2+\cdots+\lambda_i\ .$$
\end{corollary}

Now, it is an easy exercise to derive the following two convolution identities:
\begin{equation*}
\begin{bmatrix}r+m\cdot n\\r+m\cdot n-k\end{bmatrix}_r=\sum_{k_1+\cdots+k_m=k}\prod_{i=1}^{m}\begin{bmatrix}r+i\cdot n\\r+i\cdot n-k_i\end{bmatrix}_{r+(i-1)n}\ 
\end{equation*}
and
\begin{equation*}
\begin{Bmatrix}r+m\cdot n+k\\r+m\cdot n\end{Bmatrix}_{r+1} =\sum_{k_1+\cdots+k_m=k}\prod_{i=1}^{m}\begin{Bmatrix}r+i\cdot n+k_i\\r+i\cdot n\end{Bmatrix}_{r+1+(i-1)n}\ .
\end{equation*}

For $m=2$, by Corollaries \ref{C:3.5} and \ref{C:3.6}, we obtain the convolution identities \eqref{e:1.2} and \eqref{e:1.3}. 



\ 



\section{Some quantum convolution identities}

The $q$-binomial coefficients are $q$-analogs of the binomial coefficients and are defined by
$$\binom{n}{k}_q=
\begin{cases}
\dfrac{(q;q)_n}{(q;q)_k(q;q)_{n-k}}, & \mbox{for $k\in\{0,\dots,n\}$,}\\
0,& \mbox{otherwise,}\\
\end{cases}
$$
where
$$(a;q)_n=(1-a)(1-aq)\cdots(1-aq^{n-1})$$
is $q$-shifted factorial, with $(a;q)_0=1$.

\begin{corollary}\label{C:4.1}
	Let $k$, $m$ and $n$ be three positive integers and let $[\lambda_1,\lambda_2,\ldots,\lambda_m]$ be a composition of $n$. Then
	\begin{equation*}
	\sum_{k_1+k_2+\cdots+k_m=k}q^{a_1k_2+\cdots+a_{m-1}k_m-e_2(k_1,\ldots,k_m)}\prod_{i=1}^{m}\binom{\lambda_i}{k_i}_q=\binom{n}{k}_q,
	\end{equation*}
	where
	$$a_i=\lambda_1+\lambda_2+\cdots+\lambda_i\ .$$
\end{corollary} 

\begin{proof}
	To prove the corollary we use the well-known relation
	$$
	e_k(1,q,\ldots,q^{n-1})=q^{\binom{k}{2}}\binom{n}{k}_q\ .
	$$
	Taking into account that
	$$
	e_{k}(q^{p},\ldots,q^{n})=q^{pk}e_{k}(1,q,\ldots,q^{n-p})
	$$
	we can write
	\begin{eqnarray*}
		\prod_{i=1}^{m}e_{k_i}(q^{a_{i-1}},\ldots,q^{a_i-1})&=&\prod_{i=1}^{m}q^{a_{i-1}k_i}e_{k_i}(1,q,\ldots,q^{\lambda_i-1})\\
		&=&\prod_{i=1}^{m}q^{a_{i-1}k_i+\binom{k_i}{2}}\binom{\lambda_i}{k_i}_q\\
		&=&q^{a_0k_1+\cdots+a_{m-1}k_m+\binom{k}{2}-e_2(k_1,\ldots,k_m)}\prod_{i=1}^m\binom{\lambda_i}{k_i}_q\ ,
	\end{eqnarray*}
	with $a_0=0$.
	According to Theorem \ref{T:1}, the proof is finished.
\end{proof}

\begin{corollary}\label{C:4.2}
	Let $k$, $m$ and $n$ be three positive integers and let $[\lambda_1,\lambda_2,\ldots,\lambda_m]$ be a composition of $n$. Then
	\begin{equation*}
	\sum_{k_1+k_2+\cdots+k_m=k}q^{a_1k_2+\cdots+a_{m-1}k_m}\prod_{i=1}^{m}\binom{\lambda_i+k_i-1}{k_i}_q=\binom{n+k-1}{k}_q,
	\end{equation*}
	where
	$$a_i=\lambda_1+\lambda_2+\cdots+\lambda_i\ .$$
\end{corollary} 

\begin{proof}
	To prove the corollary, we use the relations
	$$h_k(1,q,\ldots,q^{n-1})=\binom{n+k-1}{k}_q\ ,$$ $$h_{k}(q^{p},\ldots,q^{n})=q^{pk}h_{k}(1,q,\ldots,q^{n-p})$$
	and Theorem \ref{T:1}.
\end{proof}

Replacing $n$ by $m\cdot n$, and $[\lambda_1,\lambda_2,\ldots,\lambda_m]$ by $[n,n,\ldots,n]$ in Corollaries \ref{C:4.1} and \ref{C:4.2}, we obtain 

\begin{corollary}\label{C:4.3}
	Let $k$, $m$ and $n$ be three positive integers. Then
	\begin{equation*}
	\sum_{k_1+k_2+\cdots+k_m=k}\frac{(q^{k_2+2k_3+\cdots+(m-1)k_m})^n}{q^{e_2(k_1,\ldots,k_m)}}\prod_{i=1}^{m}\binom{n}{k_i}_q=\binom{m\cdot n}{k}_q\ .
	\end{equation*}
\end{corollary}

\begin{corollary}\label{C:4.4}
	Let $k$, $m$ and $n$ be three positive integers. Then
	\begin{equation*}
	\sum_{k_1+k_2+\cdots+k_m=k}(q^{k_2+2k_3+\cdots+(m-1)k_m})^n\prod_{i=1}^{m}\binom{n+k_i-1}{k_i}_q=\binom{m\cdot n+k-1}{k}_q\ .
	\end{equation*}
\end{corollary}

We remark that the well-known $q$-Vandermonde's convolution identity
\begin{equation*}
\sum_{i=0}^{k}q^{(k-i)(t-i)}\binom{t}{i}_q\binom{n-t}{k-i}_q=\binom{n}{k}_q\ 
\end{equation*}
is the case $m=2$ and $\lambda=[t,n-t]$ in Corollary \ref{C:4.1}. Similarly, by Corollary \ref{C:4.2}, we get
\begin{equation*}
\sum_{i=0}^{k}q^{(k-i)t}\binom{t-1+i}{i}_q\binom{n-t+k-i}{k-i}_q=\binom{n+k}{k}_q\ .
\end{equation*}

For $q\to 1$, the limiting case of Corollaries \ref{C:4.1} and \ref{C:4.2} read as

\begin{corollary}\label{C:4.5}
	Let $k$, $m$ and $n$ be three positive integers and let $[\lambda_1,\lambda_2,\ldots,\lambda_m]$ be a composition of $n$. Then
	\begin{equation*}
	\sum_{k_1+k_2+\cdots+k_m=k}\prod_{i=1}^{m}\binom{\lambda_i}{k_i}=\binom{n}{k}\ .
	\end{equation*}
\end{corollary}

\begin{corollary}\label{C:4.6}
	Let $k$, $m$ and $n$ be three positive integers and let $[\lambda_1,\lambda_2,\ldots,\lambda_m]$ be a composition of $n$. Then
	\begin{equation*}
	\sum_{k_1+k_2+\cdots+k_m=k}\prod_{i=1}^{m}\binom{\lambda_i+k_i-1}{k_i}=\binom{n+k-1}{k}\ .
	\end{equation*}
\end{corollary}

On the other hand, these corollaries are immediate from Theorem \ref{T:1} because it is the case 
$$x_1=x_2=\cdots = x_n$$
in this theorem and
$$
e_k(\underbrace{1,\ldots,1}_n)=\binom{n}{k}\qquad\mbox{and}\qquad h_k(\underbrace{1,\ldots,1}_n)=\binom{n+k-1}{k}\ .
$$

Now, we can easily derive the following two identities
\begin{equation}\label{eq:4.1}
\sum_{k_1+k_2+\cdots+k_m=k}\prod_{i=1}^{m}\binom{n}{k_i}=\binom{m\cdot n}{k}\ ,
\end{equation}
\begin{equation}\label{eq:4.2}
\sum_{k_1+k_2+\cdots+k_m=k}\prod_{i=1}^{m}\binom{n+k_i-1}{k_i}=\binom{m\cdot n+k-1}{k}\ ,
\end{equation}
that can be rewritten in this way:

\begin{corollary}\label{C:4.7}
	Let $k$, $m$ and $n$ be three positive integers. Then
	\begin{equation}\label{eq:4.3}
	\sum_{\substack{\lambda\vdash k\\l(\lambda)\le m}}\binom{m}{m-l(\lambda),\lambda}\prod_{i=1}^{m}\binom{n}{i}^{t_i(\lambda)}=\binom{m\cdot n}{k}\ .
	\end{equation}
\end{corollary}

\begin{corollary}\label{C:4.8}
	Let $k$, $m$ and $n$ be three positive integers. Then
	\begin{equation}\label{eq:4.4}
	\sum_{\substack{\lambda\vdash k\\l(\lambda)\le m}}\binom{m}{m-l(\lambda),\lambda}\prod_{i=1}^{m}\binom{n+i-1}{i}^{t_i(\lambda)}=\binom{m\cdot n+k-1}{k}\ .
	\end{equation}
\end{corollary}

The identities \eqref{eq:4.3} and \eqref{eq:4.4} are the case $x_1=x_2=\cdots =x_n$ in Corollary \ref{C:1.1}. We can see that the number of terms in the left side of \eqref{eq:4.3} or \eqref{eq:4.4} is equal to the number of partitions of $k$ into no more than $m$ parts. It is clear that this number is less than the number of terms in the left side of \eqref{eq:4.1} or \eqref{eq:4.2} that is equal to the number of compositions of $k+m$ into exactly $m$ parts less than or equal to $n$. 

A new technique for proving and discovering combinatorial identities has been introduced. Many problems can be easily solved and they can often be extended.



\end{document}